\documentclass[11pt]{amsart}

\usepackage{amsmath,amsthm,amsfonts,amssymb,bm,graphicx }

\usepackage
{hyperref}
\hypersetup{colorlinks=true,citecolor=blue,linkcolor=blue,urlcolor=blue}

\theoremstyle{plain}
\newtheorem{theorem}{Theorem}
\newtheorem{lemma}[theorem]{Lemma}

\newtheorem{cor}[theorem]{Corollary}
\newtheorem{prop}[theorem]{Proposition}
\numberwithin{equation}{section}

\newcommand{\pr}{^\prime}
\newcommand{\prd}{^{\prime 2}}

\makeatletter
\DeclareRobustCommand\widecheck[1]{{\mathpalette\@widecheck{#1}}}
\def\@widecheck#1#2{%
    \setbox\z@\hbox{\m@th$#1#2$}%
    \setbox\tw@\hbox{\m@th$#1%
       \widehat{%
          \vrule\@width\z@\@height\ht\z@
          \vrule\@height\z@\@width\wd\z@}$}%
    \dp\tw@-\ht\z@
    \@tempdima\ht\z@ \advance\@tempdima2\ht\tw@ \divide\@tempdima\thr@@
    \setbox\tw@\hbox{%
       \raise\@tempdima\hbox{\scalebox{1}[-1]{\lower\@tempdima\box
\tw@}}}%
    {\ooalign{\box\tw@ \cr \box\z@}}}
\makeatother

\begin{document}

\author{V\' \i t\v ezslav Kala}
\address{Mathematisches Institut, Bunsenstr.~3-5, D-37073 G\"ottingen, Germany}
\address{Department of Algebra, Faculty of Mathematics and Physics, Charles University, Sokolov\-sk\' a 83, 18600 Praha~8, Czech Republic}
\email{vita.kala@gmail.com}

\title{Universal quadratic forms and elements of small norm in real quadratic fields}
 
\thanks{The author was supported by ERC Starting Grant 258713.}


\begin{abstract} For any positive integer $M$ we show that there are in\-finitely many real quadratic fields that do not admit $M$-ary universal quadratic forms (without any restriction on the parity of their cross coefficients).
\end{abstract}


\setcounter{tocdepth}{2}  \maketitle 

\section{Introduction} 

The study of universal quadratic forms can be said to have started in 1770 with the four square theorem of Lagrange, which one can formulate as the statement 
that the positive definite form $x^2+y^2+z^2+w^2$ is \textit{universal}, i.e., that it represents every positive integer. 
This has been followed by a large number of further results, however, most of them deal only with \textit{diagonal} forms $\sum a_ix_i^2$ or with \textit{classical} forms $\sum_{i\leq j} a_{ij}x_ix_j$ which have all cross-coefficients $a_{ij}$ even for $i\neq j$.
These common restrictions are somewhat unnatural, but usually greatly simplify the arguments and results --
compare for example the Conway-Schneeberger 15-theorem that a positive classical quadratic form is universal if and only if it represents 1, 2, \dots, 15 
with the Bhargava-Hanke \cite{BH} 290-theorem that a (possibly non-classical) form is universal if and only if it represents 1, 2, \dots, 290.

Let us also mention another general result in this direction, the 451-theorem of Rouse \cite{Rou} that an integral quadratic form represents all odd natural numbers if and only if it represents 1, 3, \dots, 451
(whose proof is conditional on a conjecture that three specific ternary forms represent all odd positive integers).

One can also consider totally positive definite forms over totally real number fields. Such a form is \textit{universal} if it represents all totally positive integers.
Again, much less is known in general than in the special case of classical forms, but for example 
Deutsch \cite{De} proved that the non-classical form $x^2+xy+y^2+z^2+zw+w^2$ is universal over $\mathbb Q(\sqrt 5)$, and
Sasaki \cite{Sa} classified all quaternary universal forms over $\mathbb Q(\sqrt{13})$ -- there are only two up to equivalence, and they are both non-classical. 

It may seem that the number of variables required by a universal quadratic form should grow with the discriminant of the (real quadratic) number field. This is not entirely true, as Kim \cite{Ki2} constructed an infinite family of fields of the form $\mathbb Q(\sqrt {n^2-1})$ admitting positive diagonal octonary universal  forms.

On the other hand, Blomer and the author \cite{BK} have recently shown that for each $M$ there exist infinitely many real quadratic fields which have no $M$-ary positive classical universal forms. 
The goal of the present short note is to extend this result to non-classical forms, to strengthen some of the statements (in particular Proposition \ref{no universal forms}), and to greatly simplify the proofs.
Our main result is the following theorem:

\begin{theorem}\label{main theorem}
For each positive integer $M$ there are infinitely many real quadratic fields $\mathbb Q(\sqrt D)$ which do not admit $M$-ary totally positive integral universal (possibly non-classical) quadratic forms.
\end{theorem}

The basic idea is the same as in \cite{BK} -- by Proposition \ref{no universal forms} it suffices to produce $M+1$ suitable elements $\alpha_i=p_i+q_i\sqrt D$ of $\mathcal O_K$, 
which we do by considering finite approximations $p_i/q_i$ to the continued fraction for $\sqrt D$.
These elements have small norms of size approximately $\sqrt D$ (see Proposition \ref{estimate norm}), 
which means that they are hard to be non-trivially represented by a quadratic form, 
but surprisingly also the lower bound saying that their norms are not too small plays a key role in the proof.
An important difference from the previous arguments is that we are considering a different class of continued fractions, which (together with the aforementioned improved estimate on the norms of $\alpha_i$) allows us to bypass most of the delicate technical arguments of \cite{BK}.
However, the present paper does not entirely supersede the previous one, as it produces a much thinner sequence of suitable values of $D$. 

The estimates on the norms of $\alpha_i$ have a surprising Corollary \ref{corollary coefficients} on the sizes of certain coefficients of the continued fraction for $\sqrt D$. It is perhaps already known, but we have not found it in the literature.

The proof of Proposition \ref{no universal forms} is formulated in the language of lattices associated to quadratic forms. It is also common to consider $\mathcal O_K$-lattices that are not necessarily free as $\mathcal O_K$-modules. 
Our arguments work in this setting as well, and we state the resulting generalization of Theorem \ref{main theorem} as Corollary \ref{corollary lattices}.

\

Although Proposition \ref{no universal forms} holds for any totally real number field, we do not know how to extend the rest of the arguments to fields $K$ of higher degree over $\mathbb Q$, leaving open the tantalizing question of what happens in general -- we are actually not aware of any results concerning universal quadratic forms over $\mathcal O_K$. 

Surprisingly, the situation is very different over rings of $S$-integers, as Collinet \cite{Co} recently proved that the sum of 8 squares is universal over $\mathcal O_K[\frac 12]$ for any number field $K$.

\section*{Acknowledgments}

I wish to thank Valentin Blomer for a careful reading of a draft of the paper and for his very useful advice and suggestions.

\section{Universal forms} 

In this section let $K$ be a totally real number field of degree $N$ over $\mathbb Q$.
Let $\sigma_1=\mathrm{id}, \sigma_2, \dots, \sigma_N: K\hookrightarrow\mathbb R$ be the (distinct) real embeddings of $K$.
The norm of $a\in K$ is then $N(a)=\sigma_1(a)\cdots \sigma_N(a)$.
We write $a \succ b$ to mean $\sigma_i(a)>\sigma_i(b)$ for all $1\leq i\leq N$; $a\succeq b$ denotes $a\succ b$ or $a=b$.

\begin{prop}\label{t4}\label{no universal forms}
Assume that there exist totally positive elements $a_1, a_2, \dots, a_M\in\mathcal O_K$ such that for all $1 \leq i \not = j \leq M$ we have that 
if $4a_ia_j \succeq c^2$ for $c \in \mathcal{O}_K$, then $c= 0$. 

Then there are no universal totally positive $(M-1)$-ary quadratic forms over $\mathcal O_K$. 
\end{prop}

This result is similar to Proposition 4 in \cite{BK}, but we do not need items (1), (2), (3) from there. 
Also note that the assumption in former item (4) was $a_ia_j \succ c^2$. This change is only needed in order to obtain the result for all forms, not only the classical ones.

The general idea of the proof is the same as before, but we still give full details (which are somewhat different) and state things more explicitly.

\begin{proof}  
Assume that $Q(x_i)=\sum_{1\leq i\leq j\leq n} a_{ij}x_ix_j$ is a universal totally positive quadratic form over $\mathcal O_K$ and let $A=(b_{ij})$ be its matrix with $b_{ii}=a_{ii}$ and $b_{ij}=b_{ji}=a_{ij}/2$. 

Let $B(x_i, y_i)=\sum_{1\leq i, j\leq n} b_{ij}x_iy_j$ be the associated bilinear form and consider the corresponding $\mathcal O_K$-lattice $L$. 
By this we mean that $L\subset \mathbb R^n$ has an $\mathcal O_K$-basis $\ell_1, \dots, \ell_n$ 
and $B(x_i, y_i)=\left\langle\sum x_i\ell_i, \sum y_i\ell_i\right\rangle$, where $\langle\cdot, \cdot\rangle$ is the usual inner product on $\mathbb R^n$. 
Note that for $x_i, y_i\in\mathcal O_K$ we have $B(x_i, y_i)\in\frac 12\mathcal O_K$.

Since $Q$ is universal, for each totally positive integer $\alpha$ the lattice $L$ contains a vector $v(\alpha)$ that represents $\alpha$, i.e., $\alpha=\langle v(\alpha), v(\alpha)\rangle$. In particular,
for $1\leq i\leq M$ there is $v_i\in L$ such that $a_i=\langle v_i, v_i\rangle$. From now on, we can work only in terms of the lattice $L$.

\

Let us show that $v_i$ and $v_j$ are orthogonal for $i\neq j$ by computing the inner product $\langle v_i, v_j\rangle=:c/2$ with $c\in\mathcal O_K$. 
Using the Cauchy-Schwarz inequality for $v_i$ and $v_j$ we get
$$\frac{c^2}4=\langle v_i, v_j\rangle^2\leq \langle v_i, v_i\rangle\langle v_j, v_j\rangle=a_ia_j.$$ 
Since $Q$ and $L$ are totally positive definite, we also get that
$4\sigma_h(a_i)\sigma_h(a_j)\geq \sigma_h(c)^2$ for all $1\leq h\leq N$, i.e.,
$4a_ia_j \succeq c^2.$ But this implies $c=0$ by assumption, proving that $v_i$ and $v_j$ are orthogonal.

We see that the lattice $L\subset \mathbb R^n$ contains $M$ pairwise orthogonal elements, and so $n\geq M$.
\end{proof}

\section{Continued fractions}

In this section we collect some useful results on continued fractions.

Let $\gamma=[u_0, u_1, \dots]$ be an infinite continued fraction of a real number $\gamma > 0$ and let $p_i/q_i=[u_0, \dots, u_i]$ be its $i$th approximation ($u_i, p_i, q_i\in\mathbb N$). Then it is well-known (and easy to see) that  
$p_{i+1}=u_{i+1}p_i+p_{i-1}$ and  $q_{i+1}=u_{i+1}q_i+q_{i-1}$ 
(with initial conditions $p_0=u_0$, $p_1=u_1u_0+1$, $q_0=1$, $q_1=u_1$)
and 
$$\left\lvert \frac{p_i}{q_i} - \frac{p_{i+1}}{q_{i+1}} \right\rvert=
\frac {1}{q_iq_{i+1}}.$$

Also note that 
$$\frac{p_{2i}}{q_{2i}}<\frac{p_{2i+2}}{q_{2i+2}}<\gamma<\frac{p_{2i+1}}{q_{2i+1}}<\frac{p_{2i-1}}{q_{2i-1}}.$$

\begin{lemma}\label{estimate fraction}
We have $$\frac {1}{(u_{i+1}+2)q_{i}^2}<\left\lvert \frac{p_i}{q_i} - \gamma \right\rvert<
\frac {1}{u_{i+1}q_{i}^2}.$$
\end{lemma}

\begin{proof}
For the upper bound we have 
$$\left\lvert \frac{p_i}{q_i} - \gamma \right\rvert<
\left\lvert \frac{p_i}{q_i} - \frac{p_{i+1}}{q_{i+1}} \right\rvert=
\frac {1}{q_iq_{i+1}}<
\frac {1}{u_{i+1}q_{i}^2}.$$

We show the lower bound by repeatedly using the recurrence for $q_j$ as follows:
\begin{displaymath}
\begin{split}
\left\lvert \frac{p_i}{q_i} - \gamma \right\rvert & >
\left\lvert \frac{p_i}{q_i} - \frac{p_{i+2}}{q_{i+2}} \right\rvert=
\left\lvert \left(\frac{p_i}{q_i} - \frac{p_{i+1}}{q_{i+1}}\right)-\left( \frac{p_{i+1}}{q_{i+1}} - \frac{p_{i+2}}{q_{i+2}}\right) \right\rvert=\\
& =
\frac {1}{q_iq_{i+1}}- \frac {1}{q_{i+1}q_{i+2}}=
\frac{u_{i+2}}{q_iq_{i+2}}=\frac{u_{i+2}}{q_i(u_{i+2}q_{i+1}+q_{i})}\geq\\
& \geq
\frac{1}{q_i(q_{i+1}+q_{i})}=\frac{1}{q_i(u_{i+1}q_i+q_i+q_{i-1})}>\frac {1}{(u_{i+1}+2)q_{i}^2}.
\end{split}
\end{displaymath}
\end{proof}

Assume now that $\gamma=\sqrt D$ with squarefree $D>0$. 
This means that 
$\gamma=\sqrt D=[k, \overline{u_1, u_2, \dots, u_{s-1}, 2k}]$ is periodic with period $s$ and that the sequence $(u_1, u_2, \dots, u_{s-1})=(u_1, u_2, \dots, u_r, \dots, u_2, u_1)$ is symmetric, i.e., 
$u_{s-i}=u_i$ and the central element $u_r$ (for $r=\lceil(s-1)/2\rceil$) appears once or twice.
Note that  $u_0=k$, $u_{si}=2k$, and $u_{si+j}=u_j$ for $i>0$ and $j\geq 0$.

Friesen proved the following theorem which says that it is often possible to find such a $D$.

\begin{theorem}\cite{Fr}\label{friesen}
Let $(u_1, u_2, \dots, u_{s-1})$ be a fixed symmetric sequence of positive integers such that
$$q_{s-2}\ \ or\ \ \frac{q_{s-2}^2-(-1)^s}{q_{s-1}}$$ is even. 
Then there exist infinitely many $k$ such that $\sqrt D=[k, \overline{u_1, u_2, \dots, u_{s-1}, 2k}]$ for a squarefree $D$. 
\end{theorem}

Note that $q_i$ are independent of $k$, and so the condition is well-defined.


\

Throughout the rest of the paper let $K=\mathbb Q(\sqrt D)$ with squarefree $D>0$.
Then $\mathcal{O}_K = \mathbb{Z}[\sqrt{D}]$ when $D \equiv 2, 3\pmod 4$ and $\mathcal{O}_K = \mathbb{Z}\left[\frac{1+\sqrt{D}}2\right]$ when $D \equiv 1\pmod 4$.
We denote the conjugate of $a\in K$ by $a^\prime$. The norm is just $N(a) = aa'$ and $a \succ b$ means that $a>b$ and $a\pr>b\pr$.

\

Let $$\alpha_i=p_i+q_i\sqrt D, \quad N(\alpha_i)=p_i^2-Dq_i^2.$$ 
The elements $\alpha_i$ clearly satisfy the recurrence $\alpha_{i+1}=u_{i+1}\alpha_i+\alpha_{i-1}$. 
Note that $\alpha_i$ is totally positive if and only if $i$ is odd.

To prove Theorem \ref{main theorem} we want to use Proposition \ref{no universal forms} where we take for $a_j$ some of the approximations $\alpha_i$ for a suitable value of $\sqrt D$.
The following proposition gives us crucial information about the approximate sizes of the norms of the elements $\alpha_i$, which we shall use several times in the following proofs. 
Note that when $D\equiv 1\pmod 4$ and $p_i$ and $q_i$ are both odd, we could consider $\alpha_i/2\in\mathcal O_K$ instead of $\alpha_i$,
which would improve the bounds by a factor of 4.

\begin{prop}\label{estimate norm} 
We have
$$\frac {2\sqrt D}{u_{i+1}+2.5}<|N(\alpha_i)|<\frac {2\sqrt D}{u_{i+1}-0.5}.$$
\end{prop}

\begin{proof}
Using Lemma \ref{estimate fraction}, we first see that 
$2q_i\sqrt D-\frac 1{u_{i+1}q_i}<p_i+q_i\sqrt D<2q_i\sqrt D+\frac 1{u_{i+1}q_i}$, and then we obtain the upper bound as follows:
$$|N(\alpha_i)|=(p_i+q_i\sqrt D)|p_i-q_i\sqrt D|<
\Bigl(2q_i\sqrt D+\frac 1{u_{i+1}q_i}\Bigr)\frac {1}{u_{i+1}q_{i}}=
\frac {2\sqrt D}{u_{i+1}}+\frac {1}{u_{i+1}^2q_{i}^2}<\frac {2\sqrt D}{u_{i+1}-0.5}.$$

The proof of the lower bound is similar:
$$|N(\alpha_i)|>
\Bigl(2q_i\sqrt D-\frac 1{u_{i+1}q_i}\Bigr)\frac {1}{(u_{i+1}+2)q_{i}}=
\frac {2\sqrt D}{u_{i+1}+2}-\frac {1}{u_{i+1}(u_{i+1}+2)q_{i}^2}>\frac {2\sqrt D}{u_{i+1}+2.5}.$$
\end{proof}

Conversely, we shall also need to know that every element of sufficiently small norm is one of the approximations $\alpha_i$.

\begin{lemma}\label{norm implies alpha}
a) Let $\mu\in\mathbb Z[\sqrt D]\setminus \mathbb Z$ be such that $0<|N(\mu)|<\frac{\sqrt D}{2}$. Then $\mu=n\alpha_i$ or $\mu=n\alpha_i\pr$ for some $i\geq 0$ and $n\in\mathbb Z$.

b) Let $D\equiv 1\pmod 4$ and $\mu\in\mathbb{Z}\left[\frac{1+\sqrt{D}}2\right]\setminus \mathbb Z$.
If $0<|N(\mu)|<\frac{\sqrt D}{8}$, then $\mu=n\alpha_i$ or $\mu=n\alpha_i\pr$ for some $i\geq 0$ and $n\in\frac 12\mathbb Z$.
\end{lemma}

\begin{proof}
The proof is essentially the same as the last part of the proof of Proposition 12 in \cite{BK}, so we give only a sketch. 

a) Let $\mu=x+y\sqrt D$ with $x, y\in\mathbb Z$ and $y\neq 0$. 
By factoring out $n=\pm \gcd(x, y)$, we can assume that $x$ and $y$ are coprime and $x>0$. Also, by replacing $\mu$ by its conjugate $\mu\pr$ if necessary, we can assume that also $y>0$.

We need to distinguish two cases depending on the sign of $N(\mu)$. 

\textit{Case 1:} $N(\mu)<0$, i.e., $x-y\sqrt D<0$ and $y^2D-x^2=|N(\mu)|<\frac{\sqrt D}{2}$.
Hence $y^2D-\frac{1}{2}\sqrt D -x^2<0$ and  
$\sqrt D$ lies between the roots of the quadratic polynomial $2y^2 T^2-T-2x^2$.
Thus 
$$\sqrt D<\frac{1+\sqrt{1+16x^2y^2}}{4y^2}<
\frac{1}{4y^2}+\frac{\sqrt{1+8xy+16x^2y^2}}{4y^2}=
\frac xy+\frac 1{2y^2}.$$
We also have $\frac xy<\sqrt D$ by assumption, and so we conclude that $|\frac xy-\sqrt D|<\frac 1{2y^2}$ which implies \cite[Theorem 184]{HW} that $\mu=\alpha_i$ for some $i$.

\textit{Case 2:} $N(\mu)>0$. This case is similar, and so we omit it (for details see \cite{BK}, proof of Proposition 12).

To prove b), just consider $2\mu\in\mathbb Z[\sqrt D]$ in part a).
\end{proof}

The two preceding results imply some basic information on the existence of large coefficients in the continued fraction for $\sqrt D$. We do not need this result in the paper, but it seems interesting in its own right.
Let us remind the reader that by $u\asymp D^\delta$ we here mean that given $\delta>0$, there are positive constants $c_\delta, d_\delta$ (independent of $u$ and $D$) such that $c_\delta D^\delta< u<d_\delta D^\delta$.

\begin{cor}\label{corollary coefficients}
Let $D$ be a squarefree positive integer with $\sqrt D=[k, \overline{u_1, u_2, \dots, u_{s-1}, 2k}]$ and let $\alpha_i$ be as above. 
Let $0<\varepsilon<1/2$ and assume that $u_i\asymp D^{1/2-\varepsilon}$ for some $i$. 
Then one of the following is true: 
\begin{enumerate}
\item $\alpha_{i-1}$ is divisible by a prime ramified in $\mathbb Q(\sqrt D)$, or
\item for each $m\in\mathbb N$ with $m\varepsilon<1/2$ there is $j(m)\in\mathbb N$ such that $u_{j(m)}\asymp D^{1/2-m\varepsilon}$.
\end{enumerate} 

If $D=P\equiv 1\pmod 4$ is a prime, then {\rm (1)} never happens, and so {\rm (2)} holds.
\end{cor}

\begin{proof}
By Proposition \ref{estimate norm} we know that $N(\alpha_{i-1})\asymp D^{\varepsilon}$.
Assume that $\alpha_{i-1}$ is not divisible by any ramified prime. Since $p_{i-1}$ and $q_{i-1}$ are coprime, it is also not divisible by any inert prime, and so the only primes which divide it are split.
The same holds also for $\alpha_{i-1}^m$, and so $n\nmid \alpha_{i-1}^m$ for all $n\geq 2$.
We have that $N(\alpha_{i-1}^m)\asymp D^{m\varepsilon}\ll D^{1/2}$ and $\alpha_{i-1}^m>1$, hence by Lemma \ref{norm implies alpha}, $\alpha_{i-1}^m=\alpha_{j(m)-1}$ for some $j(m)$.
Then again by Proposition \ref{estimate norm} we conclude that $u_{j(m)}\asymp D^{1/2-m\varepsilon}$.


Finally, if $D=P\equiv 1\pmod 4$ is a prime, then the only ramified prime is $(P)=\mathfrak p^2$ and $N(\mathfrak p)=P$. 
Hence if $\mathfrak p\mid \alpha_{i-1}$, then $P=N(\mathfrak p)\leq N(\alpha_{i-1})\asymp P^{\varepsilon}$, which is not possible.
\end{proof}

Note that the proof actually shows how the constants implicit in $u_{j(m)}\asymp D^{1/2-m\varepsilon}$ depend on $m$: There are positive constants $c_\varepsilon$, $d_\varepsilon$ such that $$c_\varepsilon^m D^{1/2-m\varepsilon}<u_{j(m)}< d_\varepsilon^m D^{1/2-m\varepsilon}.$$

\

The assumption in Corollary \ref{corollary coefficients} on the existence of 
$u_i\asymp D^{1/2-\varepsilon}$ is unknown in general, but if the class number of $\mathbb Q(\sqrt D)$ is 1, 
then it follows from the Riemann hypothesis for $L(s,\chi_D)$ that there are many elements $\alpha$ with norm $\asymp D^{\varepsilon}$, and hence coefficients $u_i\asymp D^{1/2-\varepsilon}$ for any $0<\varepsilon<1/2$ -- see the proof of Proposition 5 in \cite{BK}.

\section{The construction}

We are finally ready to construct the continued fractions which give us many elements satisfying the assumptions of Proposition \ref{no universal forms}.

Given $M\in\mathbb N$, let $u_1, \dots, u_{s-1}$ be a symmetric sequence of integers such that 
\begin{enumerate}
\item $r:=\left\lceil \frac {s-1}2\right\rceil\geq 2M+1$,
\item the assumptions of Theorem \ref{friesen} are satisfied, and
\item $u_i$ is rapidly increasing, i.e., $u_1\geq 2$ and 
for $1\leq i\leq r$ we have $u_{i+1}\geq u_i^{3}$. 
\end{enumerate}

For example, one can take $s\equiv 2\pmod 3$ and $u_i=3^{3^{i-1}}$ for $1\leq i\leq r$.

\

Let us show that in this setting, the assumption of Proposition \ref{no universal forms} is satisfied:

\begin{prop}\label{kill off-diagonal} Assume that $(1)$, $(2)$, and $(3)$ above are satisfied and let $i, j$ be odd such that $1\leq i<j\leq r$. If $4\alpha_i\alpha_j\succeq c^2$ for $c\in\mathcal O_K$, then $c=0$.
\end{prop}

\begin{proof} 
We can assume that $c>0$.
We shall first prove that it suffices to consider $c=\frac 12\alpha_h$ for some $h$, then that $h<j$, and finally $h\geq j$, obtaining a contradiction.

a) $c=\frac 12\alpha_h$: Using Proposition \ref{estimate norm} we have 
$$|N(c)|=|cc\pr|\leq 4\sqrt{N(\alpha_i)N(\alpha_j)}<\frac {8\sqrt D}{\sqrt{(u_{i+1}-0.5)(u_{j+1}-0.5)}}<\frac{\sqrt D}{8}$$ 
(in the last inequality we used that $j\geq i+2\geq 3$, and hence $u_{j+1}\geq u_4\geq 2^{3^3}$),
and so we can apply Proposition \ref{norm implies alpha}. 
If $c=n\alpha_h\pr$ ($n\in\frac 12\mathbb N$), then 
$$1<\left(\frac 12 \alpha_1\right)^2\leq (n\alpha_i)^2=c\prd\leq 4\alpha_i\pr\alpha_j\pr<1$$ 
(the first inequality follows from the definition of $\alpha_1$ and the last one from Lemma \ref{estimate fraction}), which is not possible. 
Hence $c=n\alpha_h$. Then $4\alpha_i\alpha_j\succeq c^2=n^2\alpha_h^2\succeq\left(\frac 12 \alpha_h\right)^2$, and so we can assume that $n=\frac 12$ and $c=\frac 12\alpha_h$.

b) $h<j$: The sequence $\alpha_t$ is increasing, as we have $\alpha_{t+1}=u_{t+1}\alpha_t+\alpha_{t-1}>u_{t+1}\alpha_t$.
Hence if $h\geq j$, then $\alpha_h\geq \alpha_j$ and (since $j\geq i+2$) $$\alpha_h\geq \alpha_{i+2}>u_{i+2}\alpha_{i+1}>u_{i+2}u_{i+1}\alpha_{i}\geq 16\alpha_i.$$
This implies $\left(\frac 12 \alpha_h\right)^2>4\alpha_i\alpha_j$, a contradiction.

c) $h\geq j$: By Proposition \ref{estimate norm} we have 
$$\frac {\sqrt D}{2(u_{h+1}+2.5)}<|N(c)|\leq 4\sqrt{N(\alpha_i)N(\alpha_j)}<\frac {8\sqrt D}{\sqrt{(u_{i+1}-0.5)(u_{j+1}-0.5)}},$$
and so $16^2(u_{h+1}+2.5)^2>(u_{i+1}-0.5)(u_{j+1}-0.5)>u_{j+1}-0.5$. But this is not possible if $h<j$, because the sequence $u_t$ is rapidly increasing (and $j\geq 3$).
\end{proof}

Theorem \ref{main theorem} now follows directly from Propositions \ref{no universal forms} and \ref{kill off-diagonal}. 

\

Let us note that the constructed elements $\alpha_i$ have other interesting properties, for example when $D\equiv 2, 3\pmod 4$, then $\alpha_i$ is not a sum of totally positive elements 
and is irreducible (under a mild additional assumption on the size of $k$), i.e., $\alpha_i=xy$ for $x, y\in\mathcal O_K$ implies that $x$ or $y$ is a unit.

\

As we already mentioned in the introduction, one can formulate the main Theorem \ref{main theorem} also more generally for $\mathcal O_K$-lattices that are not necessarily free. Let us first briefly review the definitions: 

For a (totally positive) number field $K$, a \emph{quadratic space} is a pair $(V, Q)$, where $V$ is a finite-dimensional $K$-vector space and $Q$ is a quadratic form on $V$. An $\mathcal O_K$-module $(L, Q)$ is a (quadratic) \emph{$\mathcal O_K$-lattice} if it is a lattice in $(V, Q)$ of full rank. 
An $\mathcal O_K$-lattice $(L, Q)$ is \emph{totally positive} if $Q$ is a totally positive quadratic form, it is \emph{integral} if $Q(v)\in\mathcal O_K$ for all $v\in L$, and it is \emph{universal} if for each $0\prec a\in\mathcal O_K$ there is $v\in L$ such that $Q(v)=a$. 
In the proof of Proposition \ref{no universal forms} we have used the fact that to each quadratic form corresponds an $\mathcal O_K$-lattice. 
Such an $\mathcal O_K$-lattice is always free $\mathcal O_K$-module, but there exist also non-free $\mathcal O_K$-lattices in the case when $\mathcal O_K$ is not a principal ideal domain.

All of the arguments, especially (the second part of) the proof of Proposition \ref{no universal forms}, still apply almost verbatim in this setting and we obtain:

\begin{cor}\label{corollary lattices}
For each positive integer $M$ there are infinitely many real quadratic fields $\mathbb Q(\sqrt D)$ which do not admit totally positive integral universal $\mathcal O_K$-lattices of rank $M$.
\end{cor}

\end{document}